\documentclass[12pt]{amsart}

\usepackage{amssymb} 
\usepackage{amsthm}
\usepackage{amscd}
\usepackage{multicol}
\usepackage[dvips]{color}
\newtheorem{thm}{Theorem}[section]
\newtheorem{lem}[thm]{Lemma}
\newtheorem{lem-dfn}[thm]{Lemma-Definition}
\newtheorem{prop}[thm]{Proposition}
\newtheorem{cor}[thm]{Corollary}

\theoremstyle{definition}
\newtheorem{defn}[thm]{Definition}
\newtheorem{exam}[thm]{Example}
\newtheorem{ex}[thm]{Example}

\newtheorem{quest}[thm]{Question}

\newtheorem*{acknowledgement}{Acknowledgement}
\theoremstyle{remark}

\newtheorem{rem}[thm]{Remark}
\numberwithin{equation}{section}

\newcommand{\thmref}[1]{Theorem~\ref{#1}}
\newcommand{\lemref}[1]{Lemma~\ref{#1}}
\newcommand{\corref}[1]{Corollary~\ref{#1}}
\newcommand{\proref}[1]{Proposition~\ref{#1}}
\newcommand{\remref}[1]{Remark~\ref{#1}}

\newcommand{\defref}[1]{Definition~\ref{#1}}

\newcommand{\exref}[1]{Example~\ref{#1}}

\newcommand{\figref}[1]{Figure~\ref{#1}}

\DeclareMathOperator{\Spec}{Spec}
\DeclareMathOperator{\spec}{Spec}

\DeclareMathOperator{\supp}{Supp}

\DeclareMathOperator{\Hom}{Hom}

\DeclareMathOperator{\Ker}{Ker}

\DeclareMathOperator{\HH}{H}

\DeclareMathOperator{\Ann}{Ann}

\DeclareMathOperator{\Soc}{Soc}
\DeclareMathOperator{\chara}{char}
\DeclareMathOperator{\ord}{ord}

\DeclareMathOperator{\pic}{Pic}

\DeclareMathOperator{\di}{div}

\DeclareMathOperator{\chr}{char}

\DeclareMathOperator{\nr}{nr}
\DeclareMathOperator{\br}{\bar r}

\newcommand{\m}{\mathfrak m}
\newcommand{\MM}{\mathfrak M}

\newcommand{\fra}{{\mathfrak a}}

\newcommand{\PP}{\mathbb P}

\newcommand{\Z}{\mathbb Z}

\newcommand{\C}{\mathbb C}

\newcommand{\bbZ}{\ensuremath{\mathbb Z}}

\newcommand{\cal}{\mathcal}
\newcommand{\cA}{\mathcal A}

\newcommand{\cO}{\mathcal O}

\renewcommand{\t}{\widetilde}

\newcommand{\ol}[1]{\overline {#1}}

\newcommand{\defset}[2]{{\left\{#1\,\left| \,#2 \right. \right\}}}
\newcommand{\X}{(X,o)}

\date{\today}

\begin{document}
\title{Gorensteinness for normal tangent cones of elliptic ideals}

\author{Tomohiro Okuma}
\address[Tomohiro Okuma]{Department of Mathematical Sciences, 
Yamagata University,  Yamagata, 990-8560, Japan.}
\email{okuma@sci.kj.yamagata-u.ac.jp}
\author{Kei-ichi Watanabe}
\address[Kei-ichi Watanabe]{Department of Mathematics, College of Humanities and Sciences, 
Nihon University, Setagaya-ku, Tokyo, 156-8550, Japan and 
Organization for the Strategic Coordination of Research and Intellectual Properties, Meiji University
}
\email{watnbkei@gmail.com}
\author{Ken-ichi Yoshida}
\address[Ken-ichi Yoshida]{Department of Mathematics, 
College of Humanities and Sciences, 
Nihon University, Setagaya-ku, Tokyo, 156-8550, Japan}
\email{yoshida.kennichi@nihon-u.ac.jp}
\thanks{TO was partially supported by JSPS Grant-in-Aid 
for Scientific Research (C) Grant Number 21K03215.
KW  was partially supported by JSPS Grant-in-Aid 
for Scientific Research (C) Grant Number 20K03522.
KY was partially supported by JSPS Grant-in-Aid 
for Scientific Research (C) Grant Number 19K03430.}

\begin{abstract} 
Let $A$ be a two-dimensional excellent normal Gorenstein local domain. 
In this paper, we characterize elliptic ideals $I \subset A$ 
for its normal tangent cone $\overline{G}(I)$ to be Gorenstein. 
Moreover, we classify all those ideals in a Gorenstein elliptic 
singularity in the characteristic zero case. 
\end{abstract}

\maketitle

\section*{Introduction}

Throughout this paper, let $(A,\m)$ be a Noetherian commutative local ring 
with the unique maximal ideal $\m$ and the residue field $k=A/\m$. 
Furthermore, we assume that $A$ is a two-dimensional excellent normal 
local domain (which is not regular) 
containing an algebraically closed field $k$ unless otherwise specified. 
Then there exists a resolution of singularities of $\Spec A$, and 
any $\m$-primary integrally closed ideal of $A$ admits a geometric 
representation. 
That is, for an $\m$-primary integrally closed ideal 
$I \subset A$, there exist a resolution of singularities 
$f \colon X \to \Spec A$ and an anti-nef cycle $Z$ on $X$ such that 
$I\mathcal{O}_X=\mathcal{O}_X(-Z)$ and $I=H^0(X,\mathcal{O}_X(-Z))$. 
Then we denote it by $I=I_Z$. 
This enables us to introduce the notion of `geometric' ideals and prove 
existence theorems of those ideals. 

\par\vspace{2mm}
For instance, in \cite{OWY1}, the authors introduced the notion of 
$p_g$-ideals.  
An integrally closed ideal $I$ is called a \textit{$p_g$-ideal} if  
$\dim_k H^1(X, \mathcal{O}_X(-Z))=p_g(A)$, where $p_g(A)=\dim_k H^1(X,\mathcal{O}_X)$ denotes the \textit{geometric genus} of $A$; 
see Section 2 for more details.  
Note that any integrally closed ideal is a $p_g$-ideal 
in a (two-dimensional) rational singularity (see \cite{Li}). 
Moreover, they proved that any local ring (resp. Gorenstein) 
$A$ has a $p_g$-ideal (resp. a good $p_g$-ideal) in \cite[Theorem 4.1]{OWY1}, 
where an ideal $I$ is good if and only if $I^2=QI$ and $I=Q\colon I$ for some minimal reduction.  

\par 
In \cite{OWY4}, the authors introduced the following two 
normal reduction numbers. 
For any $\m$-primary integrally closed ideal $I \subset A$ and 
its minimal reduction $Q$, we put  
\begin{eqnarray*}
\nr(I) &=& \min\{r \in \bbZ_{\ge 1} \,|\,  \overline{I^{r+1}}=Q \overline{I^r}\}, \\
\br(I)&=&\min\{r \in \bbZ_{\ge 1} \,|\, \overline{I^{n+1}}=Q \overline{I^n} 
\; \text{for all $n \ge r$}\}. 
\end{eqnarray*}
Then $I$ is a $p_g$-ideal if and only if $\br(I)=1$. 

\par \vspace{2mm}
The notion of $p_g$-ideal is 
analogous  to ideals in rational singularities 
in the   theory of 2-dimensional normal singularities. 
Analogous to elliptic singularities,  
E. Rossi and the authors  introduced the notions of 
\textit{elliptic ideals} and \textit{strongly elliptic ideals} in \cite{ORWY},
in terms of \textit{normal reduction numbers},
and discussed characterizations and an existence of those ideals. 
An ideal $I$ is an elliptic ideal if and only if 
$\br(I)=2$. 
Moreover, $I$ is a strongly elliptic ideal if and only if 
it is an elliptic ideal with $\ell_A(\overline{I^2}/QI)=1$, where 
$\ell_A(M)$ denotes the length of an $A$-module $M$.  
This notion characterizes a strongly elliptic singularity, 
that is, $p_g(A)=1$. Namely, the ring $A$ is a strongly elliptic singularity
if and only if any $\m$-primary integrally closed ideal is either a $p_g$-ideal 
or a strongly elliptic ideal. 
Similarly, the first-named author showed that 
any $\m$-primary integrally closed ideal $I \subset A$ 
in an elliptic singularity $A$ is either a $p_g$-ideal or an elliptic ideal (\cite{Ok}). 

\par \vspace{2mm}
For an ideal $I \subset A$, we put 
\[
G(I)=\bigoplus_{n=0}^{\infty} I^n/I^{n+1}, \qquad 
\overline{G}(I)=\bigoplus_{n=0}^{\infty} \overline{I^n}/\overline{I^{n+1}}. 
\] 
Then $G(I)$ (resp. $\overline{G}(I)$) 
is called the associated graded ring (resp. 
the normal tangent cone) of $I$. 
Many authors have studied the Cohen-Macaulayness and 
the Gorensteinness of $G(I)$. 
But we do not know many properties of $\overline{G}(I)$
although it seems to be important from geometric point of view. 
For instance, if $I$ is a $p_g$-ideal, then 
$\overline{G}(I)=G(I)$ is Cohen-Macaulay 
 (cf. \cite[Theorem 4.1]{OWY2}), 
and $\overline{G}(I)$ is Gorenstein if and only if it is good 
(see Proposition \ref{pgGor}).
Also, if $\overline{G}(I)$ is Cohen-Macaulay, then 
$\nr(I)=\br(I)$ holds true (see Lemma \ref{CMnrbr}). 
So it is natural to ask the following question. 

\par \vspace{2mm} \par \noindent 
{\bf Question.}
Let $I$ be an elliptic ideal. 
\begin{enumerate}
\item When is $\overline{G}(I)$ Cohen-Macaulay? 
\item When is $\overline{G}(I)$ Gorenstein? 
\end{enumerate}

\par \vspace{3mm}
The first question has a positive answer. 
Namely, if $I$ is an elliptic ideal, then $\overline{G}(I)$ is Cohen-Macaulay. 
One of the main purposes of this paper is to give an answer to 
the second question.

\par \vspace{2mm} \par \noindent 
{\bf Theorem (see Theorem $\ref{Main}$).}
Assume that $A$ is Gorenstein and $I=I_Z \subset A$ is 
an elliptic ideal. 
For any minimal reduction $Q$ of $I$,  
the following conditions are equivalent$:$
\begin{enumerate}
\item $\overline{G}(I)$ is Gorenstein. 
\item $Q \colon I=Q+\overline{I^2}$. 
\item $\ell_A(\overline{I^2}/QI)=\ell_A(A/I)$. 
\item $\overline{e}_2(I)=\ell_A(A/I)$, where  $\overline{e}_2(I)$ is 
the second normal Hilbert coefficient of $I$.  
\item 
$KZ=-Z^2$, that is, $\chi(Z)=0$, 
where $K$ denotes the canonical divisor on $X$. 
\end{enumerate}
When this is the case, $\ell_A(A/I) \le p_g(A)$.  
\par \vspace{2mm} 
In Section 2, we give a proof of the theorem above and 
several examples. 
In Section 3, we give a complete classification of elliptic ideals 
in a Gorenstein elliptic singularities which satisfies the theorem  in the characteristic zero case.

\section{Preliminaries}
 
Throughout this section, let $(A,\m)$ be an excellent  two-dimensional 
normal local domain  containing an algebraically closed field 
$k= A/\m$, where $\m$ denotes the unique maximal ideal of $A$. 
\par 
Let $I \subset A$ be an $\m$-primary ideal. 
An element $a \in A$ is said to be integral over $I$ if there exists a 
monic polynomial 
\[
f(X)= X^n +c_1X^{n-1}+\cdots + c_n \;\; (c_i \in I^i,\, i=1,2,\ldots,n)
\]
such that $f(a)=0$. 
The ideal of all elements that are integral over $I$ is called the integral closure of $I$, denoted by $\overline{I}$. 
We assume that $I$ is integrally closed (i.e. $\overline{I}=I$).  
The ideal $I$ is said to be normal if $\overline{I^n}=I^n$ for 
$n \in \Z_{\ge 1}$.  
A divisor on a resolution space whose support is contained in the exceptional set is called a {\em cycle}. 
Then in our situation, there exist an resolution of singularities $f \colon X \to \Spec A$ and an anti-nef cycle $Z$ on $X$ such that 
$I\mathcal{O}_X=\mathcal{O}_X(-Z)$ and $I=H^0(X,\mathcal{O}_X(-Z))$. 
Then we say that $I$ is represented by $Z$ on $X$, and denote it by 
$I=I_Z$. 
Note that if $Q$ is a minimal reduction of $I$ generated by $s,t\in I$, then we have a surjection $s\cO_X\oplus t\cO_X \to Q\cO_X=I\cO_X=\cO_X(-Z)$, and hence $\cO_X(-Z)$ is generated (cf. the proof of \cite[(6.2)]{Li}). 
Conversely, if $\cO_X(-Z)$ is generated, then the ideal $H^0(\cO_X(-Z))\subset A$ is represented by $Z$.

\par 
In what follows, let $f \colon X \to \Spec A$ be a resolution of 
singularities and $I=I_Z$ for some anti-nef cycle $Z$ on $X$.  
We recall the following definition of $q(I)$ as follows. 

\begin{defn}[\textrm{cf. \cite[Definition 2.1]{OWY4}}] \label{qkI}
For $I=I_Z$ and an integer $k \ge 0$  we define 
\[
q(kI) = \dim_k H^1(X,\mathcal{O}_X(-kZ)). 
\]
Then $q(kI)$ is a non-negative integer for every $k \ge 1$, and 
there exists a decreasing sequence 
\[
p_g(A) \ge q(I) \ge q(2I) \ge q(3I) \ge \cdots \ge 0,  
\]
where $p_g(A)=q(0I)=\dim_k H^1(X,\mathcal{O}_X)$ 
denotes the {\it geometric genus} of $A$ (see \cite[2.5]{OWY1}). 
Note that $p_g(A)$ is independent of the choice of $X$. 
We put $q(\infty I):=q(kI)$ for sufficiently large $k$. 
\end{defn}

It is known that the geometric genus can also be represented in terms of regular $2$-forms (cf. \cite[4.4]{karras}): 
\[
p_g(A)=\ell_A(K_A/H^0(\cO_X(K_X))),
\]
where $K_A \cong H^0(X\setminus E, \cO_X(K_X))$ is the canonical module of $A$.

\par
The following formula is called (Kato's) Riemann-Roch formula. 
The result was proved in \cite{kato} in the complex case, but it holds in any characteristic \cite{WY}. 

\begin{thm}[\textbf{Kato's Riemann Roch formula}] \label{RRformula}
Assume $I=I_Z$. 
Then we have 
\[
\ell_A(A/I)+q(I)=-\dfrac{Z^2+KZ}{2}+p_g(A),
\]
where $K=K_X$ denotes the canonical divisor on $X$. 
\end{thm}

\par \vspace{2mm}
Next, we recall the normal Hilbert coefficients and their geometric representation. 

\begin{prop}[\textrm{cf. \cite[Theorem 3.2]{OWY2}}]   \label{NorCoeff}
For $I = I_Z$,  let $\overline{P}_I(n)$ be the \textit{normal Hilbert-polynomial }of $I$$:$
\[
\overline{P}_I(n)=\bar e_0(I) {n+2 \choose 2} - \bar e_1(I)
{n+1 \choose 1} + \bar e_2(I).
\]
Then 
\begin{enumerate}
\item $\overline{P}_I(n)=\ell_A(A/\ol{I^{n+1}})$ for all $n \ge p_g(A)-1$. 
\vspace{1mm}
\item 
$\bar e_0(I)=e_0(I)=-Z^2$.
\vspace{1mm}
\item 
$\bar e_1(I)- \bar e_0(I) + \ell_A(A/I) =p_g(A) - q(I)$. 
\vspace{1mm}
\item $\bar e_2(I)= p_g(A)-q(nI)=p_g(A)-q(\infty I)$ for all $n \ge p_g(A)$. 
\end{enumerate}
\end{prop}

\par \vspace{2mm}
For any ideal $I$, let $Q$ be its minimal reduction. 
Then we put 
\[
r_Q(I)=\min\{r \in \Z_{\ge 1}\,|\, I^{r+1}=QI^r\}. 
\]
In general, $r_Q(I)$ depends on the choice of $Q$. So
we define 
\[
r(I)=\min\{r_Q(I) \,|\, \text{$Q$ is a minimal reduction of $I$}\}. 
\]
Then $r(I)$ is called a \textit{reduction number} of $I$.  
We now recall the definition of two normal reduction numbers. 

\begin{defn}[\textrm{cf. \cite[Definition 2.5]{OWY4}}] \label{NorRedNum}
For $I=I_Z$ and its minimal reduction $Q$, 
we put
\begin{eqnarray*}
\nr(I)&=& \min\{r \in \mathbb{Z}_{\ge 1}  \,| \,  \overline{I^{r+1}}=Q\overline{I^r} \},\\
\br(I) &=& \min\{r \in \mathbb{Z}_{\ge 1}  \,| \,  \overline{I^{n+1}}=Q\overline{I^n} \; \text{for all $n \ge r$}\}.
\end{eqnarray*}

\par 
A positive integer $\nr(I)$ (resp. $\br(I)$) is called 
the {\it relative normal reduction number} 
(resp. the {\it normal reduction number}) 
\end{defn}

\begin{rem}[\textrm{cf. \cite[Lemma 4.1]{Hun}}]
Let $Q,Q'$ be minimal reductions of $I$. 
Then $\overline{I^{n+1}}=Q\overline{I^n}$ 
if and only if  $\overline{I^{n+1}}=Q'\overline{I^n}$.  
\end{rem}

\par \vspace{2mm}
Two normal reduction numbers can be calculated in terms of $\{q(kI)\}$. 

\begin{prop}[\textrm{cf. \cite[Proposition 2.2]{OWY5}}] \label{q(nI)formula}
For $I=I_Z$, the following statements hold. 
\begin{enumerate}
\item 
For any integer $n \ge 1$, we have 
\[
2 \cdot q(nI) + \ell_A(\overline{I^{n+1}}/Q\overline{I^n})
=q((n+1)I)+q((n-1)I).  
\]
 \item We have
\begin{eqnarray*}
\nr(I) &=& \min\{n \in \bbZ_{\ge 1} \,| \, 
   q((n-1)I)- q(nI) = q(nI) - q((n+1)I)  \},\\
\br(I) &=& \min\{n \in \bbZ_{\ge 1} \,|\, q((n-1)I)=q(nI)  \}.
\end{eqnarray*}
\end{enumerate}
\end{prop}

\par \vspace{2mm}
For instance, if $\br(I)=2$, then $q(I)=q(2I)$. 
Hence we have 
\[
\ell_A(\overline{I^2}/QI)=p_g(A)-q(I)
\]
for any minimal reduction $Q$ of $I$. 

\begin{cor} \label{pgbr}
For $I=I_Z$, the following statements hold. 
\begin{enumerate}
\item $\br(I) \le p_g(A)+1$. 
\item If equality holds in $(1)$, then $\nr(I)=1$. 
\end{enumerate}
\end{cor}
\begin{proof}
(1) Put $g=p_g(A)$. Then 
\[
g=p_g(A) = q(0I) \ge q(I) \ge q(2I) \ge \cdots \ge q(gI) \ge q((g+1)I)
\]
is a decreasing sequence of non-negative integers of length $g+2$. 
Hence there exists an integer $n$ with $1 \le n \le g+2$ such that 
$q((n-1)I)=q(nI)$. 
Hence $\br(I) \le n \le g+1$. 
\par \vspace{2mm}
(2) Suppose $\br(I)=g+1$. 
Then 
\[
g=p_g(A) = q(0I) > q(I) > q(2I) > \cdots > q(gI) = q((g+1)I) \ge 0
\]
yields $q(kI) = g-k$ for each $k=1,2,\ldots,g$. 
\par
If we substitute $n=1$ in (1) in \proref{q(nI)formula}, 
then we have $\ell_A(\overline{I^2}/QI)=0$. 
Hence $\nr(I)=1$. 
\end{proof}

\par 
Note that there exists an example of $A$ and an $\m$-primary 
integrally closed ideal $I$ which satisfies 
$\nr(I)=1 < \br(I)=p_g(A)+1$; 
see \cite[Example 3.9 and 3.10]{OWY5} and 
Example \ref{nr(I)<r(I)}. 

\par \vspace{2mm}
Now we recall the definition of normal tangent cones 
as the main target in this paper. 

\begin{defn} \label{def:cones}
Let $I=I_Z \subset A$ be an $\m$-primary integrally closed ideal. 
Then
\begin{enumerate}
\item  $G(I) = \oplus_{n \ge 0} I^n/I^{n+1}$ 
is called the {\it associated graded ring} or the {\it tangent cone} of $I$. 
\item $\overline{G}(I) = \oplus_{n \ge 0} \overline{I^n}/\overline{I^{n+1}}$ is called 
the {\it normal tangent cone} of $I$. 
\end{enumerate}
\end{defn}

\par 
In this paper, we want to investigate some ring-theoretic 
properties of normal tangent cones of integrally closed ideals.  
We first consider the case of $p_g$-ideals, 
which was introduced by the 
authors in \cite{OWY1}.  

\begin{defn}[\textrm{cf. \cite{OWY1}}] \label{def:pgideal}
Assume that $I = I_Z$. 
Then $I$ is called the \textit{$p_g$-ideal} if one of the following equivalent conditions is satisfied: 
\begin{enumerate}
\item $q(I)=p_g(A)$. 
\item $\bar e_1(I)-\bar e_0(I)+\ell_A(A/I)=0$. 
\item $\bar e_2(I)=0$. 
\item $\br(I)=1$. 
\item $I$ is normal and $I^2=QI$ for some 
minimal reduction $Q$ of $I$. 
\end{enumerate}
\end{defn}

\par \vspace{2mm}
Then the following results are known. 

\begin{prop}[\textrm{cf. \cite{ORWY}}] \label{pgCM}
If $I$ is a $p_g$-ideal, then $\overline{G}(I)=G(I)$ is Cohen-Macaulay. 
\end{prop}

\par \vspace{2mm}
Note that $\nr(I)=1$ does not necessarily imply 
the Cohen-Macaulayness of $\overline{G}(I)$.

\begin{exam}[\textrm{cf. \cite[Example 3.10]{OWY5}}] 
\label{nr(I)<r(I)}
Let $K$ be a filed of $\chara K \ne 2,3$, and 
let $A=K[[xy,xz,yz,y^2,z^2]]$ with 
$x^2=y^6+z^6$. 
That is, $A=K[[a,b,c,d,e]]/\fra$, where $\fra$ is generated 
by the following polynomials:
\[
\begin{array}{lll}
a^2-d^4-de^3, & ab-cd^3-ce^3, & b^2-d^3e-e^4 \\
ac-bd, & ae-bc, & de-c^2 
\end{array}
\]
Then $A$ is an excellent normal local domain with $\dim_k \m/\m^2=e_0(\m)+\dim A-1=5$. 
\par \vspace{1mm}
If we  put $I=(xy,xz,yz,y^2,z^4)=(a,b,c,d,e^2) \supset Q=(c,d-e^2)$, 
then $I=I_Z$ and $1=\nr(I) < \br(I)=3$. 
In particular, $\overline{G}(I)$ is not Cohen-Macaulay 
by the following lemma. 
\par \vspace{1mm}
On the other hand, $G(I)$ is Cohen-Macaulay because $I^2=QI$.  
\end{exam}

\begin{lem}\label{CMnrbr}
Assume that $\overline{G}(I)$ is Cohen-Macaulay. Then
\begin{enumerate}
\item $\nr(I)=\br(I)$ holds true.  
\item If $A$ is not a rational singularity, that is, $p_g(A) \ge 1$, then 
$\br(I) \le p_g(A)$. 
\end{enumerate}
\end{lem}

\begin{proof}
(1) Assume that $\overline{G}(I)$ is Cohen-Macaulay
 and put $\nr(I)=r \ge 1$. 
Take a minimal reduction $Q$ of $I$. 
Then we have $\overline{I^{n+1}} \subset \overline{I^{r+1}}=
Q\overline{I^r} \subset Q$ for every integer $n \ge r$.
Since $\overline{G}(I)$ is Cohen-Macaulay, we get $\overline{I^{n+1}} \subset Q \cap \overline{I^{n+1}}=Q\overline{I^n}$ for such an integer $n$ (cf. \cite{FormRing}, \cite{GN}). 
Therefore $\br(I) \le r=\nr(I)$, as required.  
\par \vspace{1mm}
(2) The assertion follows from (1) and Corollary \ref{pgbr}. 
\end{proof}

\par \vspace{2mm}
An $\m$-primary ideal $I$ is called {\it good} if $I^2=QI$ and 
$I=Q:I$ for some (every) minimal reduction $Q$ of $I$. 

\begin{prop} \label{pgGor}
Assume $A$ is Gorenstein and $I=I_Z$ is a $p_g$-ideal, that is, $I=I_Z$ and 
$\br(I)=1$.  
Then the following conditions are equivalent$:$ 
\begin{enumerate}
\item $\overline{G}(I) (=G(I))$ is Gorenstein. 
\item $I$ is good. 
\item $2 \cdot \ell_A(A/I)=e_0(I)$ 
$(\text{i.e.}\;\ell_A(A/I)=\ell_A(I/Q))$. 
\item $KZ=0$, where $K$ is a canonical divisor on $X$. 
\end{enumerate}
\end{prop}

\begin{proof}
$(1)\Longleftrightarrow (2) \Longleftrightarrow (3)$ follows from 
\cite[Theorem 6.1]{GIW}. 
$(3) \Longleftrightarrow (4)$ follows from Kato's Riemann-Roch formula
(Theorem \ref{RRformula}). 
\end{proof}

\par \vspace{2mm}
Next, we recall the notion of elliptic ideals and strongly elliptic ideals. 

\begin{defn}[\textrm{cf. \cite[Theorem 3.2]{ORWY}}] \label{def:elliptic} 
An ideal $I=I_Z$ is called the \textit{elliptic ideal}
if one of the following equivalent conditions is satisfied: 
\begin{enumerate}
\item $q(I) =q(\infty I) < p_g(A)$.
\item $\bar e_1(I)-\bar e_0(I)+\ell_A(A/I)= \bar e_2(I)>0$. 
\item $\br(I)=2$. 
\end{enumerate}
\end{defn}

\begin{defn}[\textrm{cf. \cite[Theorem 3.9]{ORWY}}] \label{def:SEideal}
An ideal $I=I_Z$ is called the \textit{strongly elliptic ideal}
if one of the following equivalent conditions is satisfied: 
\begin{enumerate}
\item $p_g(A)-1=q(I) =q(\infty I)$. 
\item $\bar e_1(I)-\bar e_0(I)+\ell_A(A/I)=1$ and $\nr(I)=\br(I)$. 
\item $\bar e_2(I)=1$.  
\item $\br(I)=2$ and $\ell_A(\overline{I^2}/QI)=1$.  
\end{enumerate}
\end{defn}

\begin{prop}[\textrm{cf. \cite[Theorem 3.2]{ORWY}}] \label{ellCM}
If $I$ is an elliptic ideal, then $\overline{G}(I)$ is Cohen-Macaulay. 
\end{prop}

\begin{rem}
If $A$ is an analytically unramified Cohen-Macaulay local domain 
containing a field with $\dim A=2$ and $\br(I)=2$, then the Cohen-Macaulayness of $\overline{G}(I)$ also follows from 
\cite[Theorems 4.4, 4.5 and Theorem of p.317]{Hun}. 
\end{rem}

\par \vspace{2mm}
So it is natural to ask the following question. 

\begin{quest}
Let $I \subset A$ be an elliptic ideal. 
When is $\overline{G}(I)$ Gorenstein?
\end{quest} 

\par \vspace{2mm}
In the next section, we give an answer as the main result in this paper. 
Before doing that, we recall the following criterion for Gorensteinness of $G(I)$. 
Note that if $G(I)$ or $\overline{G}(I)$ is Gorenstein, then so is $A$ (see Proposition 1.6 and 5.17 of \cite{TW}).

\begin{thm}[\textrm{cf. \cite[Theorem 1.4]{GI}}]
Assume that $A$ is a Gorenstein local ring of any dimension. 
Let $I \subset A$ be an $\m$-primary ideal with $r(I)=r$, 
and let $Q$ be a minimal reduction of $I$.  
Suppose that $G(I)$ is Cohen-Macaulay. 
Then the following conditions are equivalent$:$
\begin{enumerate}
\item $G(I)$ is Gorenstein. 
\item $I^r=Q^r \colon I^r$. 
\end{enumerate} 
\end{thm}


\section{Gorensteinness for normal tangent cone of elliptic ideals}

In this section, we keep the notation as in the previous section. 
For example,  $(A,\m)$ denotes an excellent  two-dimensional 
normal local domain  containing an algebraically closed field 
$k= A/\m$.
The main aim of this section is to prove the following theorem. 

\begin{thm} \label{Main}
Assume that $A$ is Gorenstein and $\br(I)=2$. 
For any minimal reduction $Q$ of $I$,  
the following conditions are equivalent$:$
\begin{enumerate}
\item $\overline{G}(I)$ is Gorenstein. 
\item $Q \colon I=Q+\overline{I^2}$. 
\item $\ell_A(\overline{I^2}/QI)=\ell_A(A/I)$. 
\item $\overline{e}_2(I)=\ell_A(A/I)$. 
\item $KZ=-Z^2$, that is, $\chi(Z):=-\dfrac{KZ+Z^2}{2}=0$
\end{enumerate}
When this is the case, $\ell_A(A/I) \le p_g(A)$.
\end{thm}

We use the following lemma to prove the theorem above.

\begin{lem}
[cf. Huneke {\cite[Theorem, p.317]{Hun}},   Itoh {\cite[Theorem 1]{It1}}]
\label{l:H-I}
Let $(A,\m)$ be a Cohen-Macaulay local ring, and $I$ an $\m$-primary integrally closed ideal.
If $Q$ is a minimal reduction of $I$, then $Q \cap \overline{I^2}=QI$.
\end{lem}

\begin{proof}[Proof of \thmref{Main}]
By Proposition \ref{ellCM},
$\overline{G}(I)$ is Cohen-Macaulay. 

\par \vspace{1mm}
Let $Q=(a,b)$ be a minimal reduction of $I$ and fix it. 
Then $a^{*}, b^{*} \in \overline{G}(I)$ forms a homogeneous regular sequence of degree $1$ and we get 
\[
B:=\overline{G}(I)/(a^{*},b^{*}) \cong
A/I \oplus I/(Q+\overline{I^2}) \oplus (Q+\overline{I^2})/Q. 
\]
Thus $\overline{G}(I)$ is Gorenstein if and only if $B$ is Gorenstein. 
$B$ is an Artinian local ring with the maximal ideal 
$ \MM_B = \m/I \oplus B_1 \oplus B_2$ and we denote $\Soc(B) = \Ann_B 
(\MM_B)$.   

\par \vspace{2mm} 
$(2) \Longrightarrow (1):$ 
It suffices to show $\dim_k \Soc(B)=1$.  
Let $x^{*}$ be a homogeneous element  in $\Soc(B)$. 
Namely, for $x \in \overline{I^m}$, we denote $x^{*}$ 
the image of $x$ in $B_m$. 
\par \vspace{1mm} \par \noindent
\underline{\bf The case where $x^{*} \in B_0$:}
\par \vspace{1mm}
Since $x^{*} B_2 =0$ in $B_2$, we have
$x \in Q \colon \overline{I^2}=Q \colon (Q+\overline{I^2})=Q:(Q:I)=I$. 
Hence $x^{*}=0$ in $B_0$. 
\par \vspace{1mm} \par \noindent
\underline{\bf The case where $x^{*} \in B_1$:}
\par \vspace{1mm}
Since $x^{*}B_1=0$ in $B_2$, we have $x \in Q:I=Q+\overline{I^2}$. 
Hence $x^{*}=0$ in $B_1$.

\par \vspace{1mm} \par \noindent
\underline{\bf The case where $x^{*} \in B_2$:}
\par \vspace{1mm}
Since $x \in Q:\m$, we have $x^{*} \in \Soc((Q:I)/Q) 
=(Q\colon \m)/Q\cong k$. 
\par \vspace{1mm}\par \noindent
Therefore $\dim_K \Soc(B)=1$ and thus $B$ is Gorenstein. 

\par \vspace{2mm} 
$(1) \Longrightarrow (2) \Longleftrightarrow (3):$ 
\lemref{l:H-I} yields  
\[
\ell_A(B_2)=\ell_A(Q+\overline{I^2}/Q) 
=\ell_A( \overline{I^2}/Q \cap \overline{I^2})=
\ell_A( \overline{I^2}/QI) (\ne 0). 
\]
From Matlis duality, we have 
\[
\ell_A(Q:I/Q)=\ell_A(\Hom_A(A/I,A/Q))=\ell_A(K_{A/I})=\ell_A(A/I)=\ell_A(B_0) \; (\ne 0). 
\]
As $I\overline{I^2} \subset \overline{I^3} = Q \overline{I^2} \subset Q$, 
we have $\overline{I^2} \subset Q:I$. 
In particular, $Q+\overline{I^2} \subset Q:I$. 
Hence
\[
\ell_A(A/I) - \ell_A(\overline{I^2} /QI) = \ell_A(B_0) - \ell_A(B_2)
=\ell_A(Q:I/Q+\overline{I^2})\ge 0. 
\]
If $B$ is Gorenstein, then $\ell_A(B_0)-\ell_A(B_2)=0$.  
Thus the implication follows from here. 

\par \vspace{2mm} 
$(3) \Longleftrightarrow (4) \Longleftrightarrow (5):$ 
Since $\br(I)=2$, we have $q(I)=q(2I)=q(\infty I)$. 
Then by Proposition \ref{q(nI)formula}, we have 
\[
\ell_A(\overline{I^2}/QI) =p_g(A)-q(I). 
\]
Moreover, Proposition \ref{NorCoeff} implies
\[
\bar e_2(I)=p_g(A)-q(\infty I) = p_g(A)-q(I)=\ell_A(\overline{I^2}/QI).
\]
On the other hand, Kato's Riemann-Roch formula yields 
\[
\ell_A(A/I)+q(I)=\chi(Z)+p_g(A). 
\]
Hence 
\[
\chi(Z)=\ell_A(A/I)-\{p_g(A)-q(I)\}=\ell_A(A/I)-\ell_A(\overline{I^2}/QI)=\ell_A(A/I)-\bar e_2(I). 
\]
The assertion follows from here. 
\end{proof}

\par \vspace{2mm}
Using the theorem above, we ask the following question:

\begin {quest} 
When is $\overline{G}(\m)$ Gorenstein? 
\end{quest}

\par 
If $\br(\m)$ is small enough, then $\overline{G}(\m)$ is Gorenstein. 
Precisely speaking, if $\br(\m) \le 2$, then $\overline{G}(\m)$ is Gorenstein. 
But there exists an example of $A$ for which $\overline{G}(\m)$ is \textit{not} Gorenstein
(see Example \ref{ex355}).  
  
\begin{prop} \label{GmGor}
Assume $A$ is Gorenstein which is not regular. 
\begin{enumerate}
\item If $\br(\m) =1$, then $\overline{G}(\m)=G(\m)$ is Gorenstein. 
\item If $\br(\m) =2$, then $\overline{G}(\m)$ is Gorenstein. 
\item If $p_g(A) \le 2$, then  $\overline{G}(\m)$ is Gorenstein. 
\end{enumerate}
\end{prop}

\begin{proof}
(1) Suppose $\br(\m)=1$. 
Since $\m$ is a $p_g$-ideal, then $\m^2=Q\m$ and $\m$ is normal. 
As $A$ is Gorenstein, we have $e_0(\m)=2$. 
In particular, $A$ is a hypersurface. 
Thus $\overline{G}(\m)=G(\m)$ is Gorenstein. 
\par \vspace{1mm} 
(2) Suppose $\br(\m)=2$.
Take a minimal reduction $Q$ of $\m$. 
Then $Q + \ol{\m^2} \ne Q$  by \lemref{l:H-I}.  
Since $Q \subsetneq Q+\overline{\m^2} \subset Q:\m$ and 
$\ell_A(Q \colon \m/Q)=1$, we have $Q:\m=Q+\overline{\m^2}$.  
Then Theorem \ref{Main} implies that $\overline{G}(\m)$ is Gorenstein. 
\par \vspace{1mm} 
(3) If $p_g(A) \le 1$, then $\br(\m) \le p_g(A)+1 \le 2$.  
If $p_g(A)=2$, then $A$ is an elliptic singularity by \cite[Theorem B]{yau.max}. 
Thus \cite[Theorem 3.3]{Ok} implies $\br(\m) \le 2$.  
In both cases, the assertion follows from (1), (2). 
\end{proof}

A local ring $A$ which satisfies $p_g(A)=1$ is called 
a {\it strongly elliptic singularity}. 
Then any $\m$-primary integrally closed ideal $I$ 
is either a $p_g$-ideal or a strongly elliptic ideal; see \cite{ORWY}.

\begin{cor}
Assume that $A$ is Gorenstein. 
Let $I$ be a strongly elliptic ideal. 
Then $\overline{G}(I)$ is Gorenstein if and only if $I=\m$.  
\end{cor}

\par \vspace{2mm}
A Gorenstein local ring $A$ is called a \textit{minimally elliptic singularity} if it is a strongly elliptic singularity. 

\begin{cor}
Assume that $A$ is a minimally elliptic singularity. 
Let $I$ be an integrally closed $\m$-primary ideal. 
Then $\overline{G}(I)$ is Gorenstein 
if and only if one of the following cases occurs:
$($a$)$ $I=\m$ $($b$)$ $I$ is normal and good. 
\end{cor}

\par \vspace{2mm}
In order to give several examples, we recall the following lemma. 

\begin{lem} [cf. \cite{OWY4}]
Let $2 \le a \le b \le c$ be integers, and put $R=\C[x,y,z]_{(x,y,z)}$, 
and $A=R/(x^a+y^b+z^c)$. 
Then $A$ is an excellent normal Gorenstein local domain 
of $\dim A=2$. 
\begin{enumerate}
\item $p_g(A)
=\sharp\{(i,j,k) \in \bbZ_{\ge 0} \,|\,0 \le ibc+jac+kab \le a(A) \}$, 
where $a(A)=abc-(ab+bc+ca)$; this formula follows from 
Theorem~$\ref{t:pg-f}$ below. 
\item $\br(\m)= \lfloor \frac{(a-1)b}{a} \rfloor$.  
\end{enumerate}
\end{lem}

For the convenience of readers, 
we recall the formula for $p_g(S)$ in the 
case $S$ is a normal Gorenstein graded ring.

\begin{thm}
[see {\cite[2.10]{Wt_rat}, \cite[3.3]{Wt}}]
\label{t:pg-f} 
Let $S=\bigoplus _{i\ge 0}S_i$ be a two-dimensional normal graded ring finitely generated over $S_0=k$ and let $S_+ = \bigoplus _{i\ge 1}S_i$.
Then the geometric genus of $(S, S_+)$ 
$($or its completion by $S_+$$)$ is given by 
\[
p_g(S)=\sum_{i=0}^{a(S)} \dim_k S_i,
\]
where $a(S)$ is the $a$-invariant of Goto--Watanabe $($\cite{GW}$)$.
\end{thm}

\begin{exam} \label{ex:Brieskorn}
Assume that $A$ is Gorenstein but not regular, 
and let $g \ge 1$ be an integer. 
\begin{enumerate}
\item The case where $\br(\m)=1$. 
\begin{enumerate}
\item If $A$ is a rational singularity, that is, $p_g(A)=0$, then $\br(\m)=1$ and $\bar G(\m)=G(\m)$ is a hypersurface. 
\item If $A=\C[x,y,z]_{(x,y,z)}/(x^2+y^3+z^{6g+1})$,  
then $p_g(A)=g$, $\br(\m)=1$ and 
$\overline{G}(\m)=G(\m) \cong K[X,Y,Z]/(X^2)$ is Gorenstein. 
\end{enumerate}
\item  The case where $\br(\m)=2$. 
\par \vspace{1mm}
\begin{enumerate}
\item If $e_0(\m) \ge 3$, then $\m$ is a strongly elliptic ideal and 
$I=\m$ satisfies the condition of Theorem \ref{Main}(2). 
\par 
For instance, 
if $A=\C[x,y,z]_{(x,y,z)}/(x^3+y^3+z^{3g})$, 
then $p_g(A)=g$, $\br(\m)=2$  and 
\[
\overline{G}(\m) =G(\m) \cong 
\left\{
\begin{array}{ll}
\C[X,Y,Z]/(X^3+Y^3+Z^3) & \; \text{(if $g=1$)}; \\[1mm]
\C[X,Y,Z]/(X^3+Y^3)  & \; \text{(otherwise)}.
\end{array}
\right.
\]
\item If $A=\C[x,y,z]_{(x,y,z)}/(x^2+y^4+z^{4g})$, 
then $p_g(A)=g$, $\br(\m)=2$  and  $G(\m) \cong \C[X,Y,Z]/(X^2)$. 
Moreover, we have 
\[
\overline{G}(\m) \cong 
\left\{
\begin{array}{ll}
\C[X,Y,Z]/(X^2+Y^4+Z^4) & \; \text{(if $g=1$)}; \\[1mm]
\C[X,Y,Z]/(X^2+Y^4)  & \; \text{(otherwise)}.
\end{array}
\right.
\]
\end{enumerate}
\end{enumerate}
\end{exam}

\par 
The next example shows that neither  $\br(\m)=3$ nor $p_g(A)=3$ 
does \textit{not} imply 
that $\overline{G}(\m)$ is Gorenstein. 

\begin{exam}\label{ex355}
Let $A=\C[x,y,z]_{(x,y,z)}/(x^3+y^5+z^5)$. 
Then 
\begin{enumerate}
\item $\br(\m)=3$. 
\item $p_g(A)=3$. 
\item $\overline{G}(\m)$ is not Gorenstein. 
In fact, $\overline{G}(\m)/(y^{*}, z^{*}) \cong \C[X,V]/(X^2,XV,V^3)$. 
\end{enumerate} 
\end{exam}

\par \vspace{2mm}
We can probably give a similar criterion as in the case of $\br(I)=2$. 
Indeed, we have several results for Brieskorn type hypersurfaces
(cf. \cite{OWY4}).   
Moreover, we have the following example. 

\begin{exam}
Let $r,n$ be integers with $r \ge 2$ and $n \ge 1$. 
Let 
\[
A=k[x,y,z]_{(x,y,z)}/(x^{r+1}+y^{r+1}+z^{(r+1)n}).
\]
Let $I=(x,y,z^n)$ and $Q=(y,z^n)$. 
Then 
\begin{enumerate}
\item $\overline{G}(I)=G(I)$ is Gorenstein. 
\item $Q \colon I=Q+I^r$. 
\item $\ell_A(A/I)=\ell_A(I^r/QI^{r-1})$. 
\end{enumerate}
\end{exam}

\begin{quest}
Does $\br(I)=2$ imply $r(I) \le 2$? 
\end{quest}

\par \vspace{1mm}
This is true for strongly elliptic ideals. 
But we do not know in the case of general elliptic ideals. 
Moreover, if this were true, then we can give a complete answer 
to the question: when is $G(I)$ Gorenstein for elliptic ideals?  

\begin{exam}
Let $A=\C[x,y,z]_{(x,y,z)}/(x^2+y^4+z^8)$ and put $I=(x,y,z^2)$ and
$Q=(y,z^2)$.  
Then 
\begin{enumerate}
\item $\br(I)=2$. 
\item $r(I)=1$ and $Q:I=I$. In particular, $I$ is not normal. 
\item $\overline{G}(I) \cong \C[X,Y,Z,W]/(X^2+Y^4+W^4, Z^2)$ is Gorenstein. 
\end{enumerate} 
\end{exam} 

\section{Existence of elliptic ideals for which $\overline{G}(I)$ is Gorenstein}

The aim of this section is to give a complete classification of 
elliptic ideals $I = I_Z$ in a Gorenstein elliptic singularity 
which satisfy the conditions of \thmref{Main}.
First, we show that such a cycle $Z$ for which $I=I_Z, \br(I)=2$ and 
$\bar{G}(I)$ is Gorenstein is characterized  in terms of
the \lq\lq elliptic sequence" on 
the minimal resolution $X$ of $\Spec(A)$ and,  
in particular the possibility of 
such $Z$ is finite and bounded by $p_g(A)$ (\thmref{t:ZcA}). 
Up to this point,  
our argument is valid in any characteristics. 
To guarantee the existence of such ideals, we need to consider \proref{p:bsz^2-1}.
In the case $- Z^2 =1$, we need to check the Picard groups of the 
exceptional curves and have to assume $\chara{k}=0$. 
We will also give  examples of singularities $A, B$ which are hypersurfaces, have the same resolution graph,
such that the number of ideals $I$ for which $\bar{G}(I)$ is Gorenstein is $p_g(A)$ for $A$ and $0$ for $B$ (see Example \ref{e:NoMaxEll}).
\par 
Let $f \colon X \to \Spec A$ be a resolution with exceptional set $E$ and let $E=\bigcup E_i$ be the decomposition into irreducible components.
A divisor $D$ on $X$ is said to be {\em nef} (resp. {\em anti-nef}) if $DE_i\ge 0$ (resp. $DE_i\le 0$) for all $E_i \subset E$.
Recall that a divisor $D$ is called a {\em cycle} if $\supp(D)\subset E$. 
It is known that $D\ge 0$ if $D$ is an anti-nef cycle. 
 For a cycle $D > 0$ on $X$, we denote by $\chi(D)$ the Euler characteristic $\chi(\cO_D) = h^0(\cO_D) - h^1(\cO_D)$ and by $Z_D$ the fundamental cycle on $\supp(D)$, i.e., $Z_D$ is the minimal cycle such that $\supp(Z_D)=\supp(D)$ and $Z_DE_i\le 0$ for all $E_i\subset \supp(D)$.
By the Riemann-Roch formula, $\chi(D)=-(D^2+D K_X)/2$, where $K_X$ denotes a canonical divisor on $X$.

\begin{defn}[Wagreich {\cite[p. 428]{W}}]
A normal surface singularity $(A,\m)$ is called an {\em elliptic singularity} if one of the following equivalent conditions  holds:
\begin{enumerate}
 \item $\chi(D)\ge 0$ for all cycles $D>0$ and $\chi(F)=0$ for some
       cycle $F>0$;
 \item  $\chi(Z_E)=0$.
\end{enumerate}
(For the proof of the implication (2) $\Rightarrow$ (1), see \cite[Corollary 4.2]{La}, \cite[(6.4), (6.5)]{T}.)
\end{defn}

\begin{defn}[Laufer {\cite[Definition 3.1 and 3.2]{La}}]
If $A$ is an elliptic singularity, then there exists a unique cycle $E_{min}$ such that $\chi(E_{min})=0$ and $\chi(D)>0$ for all cycles $D$ satisfying $0<D<E_{min}$.
The cycle $E_{min}$ is called the {\em minimally elliptic cycle}.
\end{defn}

From \cite[Proposition 3.1, Proposition 3.2, Corollary 4.2]{La}, \cite[p.428]{W}, \cite[(6.4), (6.5)]{T}, we have the following.

\begin{prop}[cf. {\cite[2.3]{o.numGell}}]
 \label{p:ellchi}
Assume that $A$ is an elliptic singularity and let $D>0$ be an arbitrary cycle on $X$. 
Then we have the following.
\begin{enumerate}
\item $\chi(D)\ge 0$.
If $\chi(D)=0$, then $D\ge E_{min}$ and $D$ is connected, i.e., $\supp (D)$ is connected.
 \item If $D$ is a connected reduced cycle without component of $E_{min}$, then $E_{min}  D\le 1$ and $\chi(Z_D)=1$.  
\end{enumerate}
\end{prop}

\begin{lem}\label{l:repMin}
Assume that $A$ is an elliptic singularity and let $I \subset A$ be an ideal represented by a cycle $Z>0$ on $X$.
Let $X_0 \to \spec A$ be the minimal resolution and $\phi: X\to X_0$ the natural morphism. If $\chi(Z)=0$, then $I=I_Z$ can be represented on $X_0$ by $\phi_*Z$. 
\end{lem}
\begin{proof}
Let $Z_0=\phi_*Z$, $F=Z-\phi^*Z_0$, and $K_{\phi}=K_X-\phi^*K_{X_0}$.
Then $K_{\phi}\equiv K_X$ on the $\phi$-exceptional set, namely, $(K_{\phi}-K_X)E_i=0$ for all $E_i\subset E$ such that $\phi(E_i)$ is a point.
Hence we have   
\[
Z^2+Z K_X = (\phi^*Z_0 + F)^2 + (\phi^*Z_0 + F)(\phi^*K_{X_0} + K_{\phi})
= Z_0^2 + Z_0 K_{X_0} + F^2 + F K_X.
\]
Since $F\equiv Z$ on the $\phi$-exceptional set and $Z$ is anti-nef, we have $F\ge 0$. 
Suppose that $F\ne 0$.
Then we have $H^1(\cO_F)=0$ and $\chi(Z)=\chi(Z_0)+\chi(F)$.
Since $A$ is elliptic and $\chi(Z)=0$, we have $\chi(Z_0)=\chi(F)=0$; however, it contradicts that $\chi(F)=h^0(\cO_F)>0$.
Therefore, $F=0$ and $ZC = 0$ for every $\phi$-exceptional curve $C$.
Hence $\cO_{X_0}(-Z_0)$ is generated.
\end{proof}
\par 
In the rest of this section, we always assume that  {\bf $A$ is Gorenstein elliptic singularity}.
By virtue of \lemref{l:repMin}, for the aim of this section, we may also assume that {\bf $f \colon X \to \Spec A$ is the minimal resolution}.

\begin{lem}\label{l:br=2} 
Under the assumption above, for  any $\m$-primary integrally closed ideal $I \subset A$, we have the following.
\begin{enumerate}
\item $\br(I) = 2$ if and only if $q(I) < p_g(A)$.  Namely, $I$ is an elliptic ideal if and only if it is not a $p_g$-ideal.
\item If $I$ is represented by a cycle on $X$, then $\br(I)=2$.
\end{enumerate}
\end{lem}
\begin{proof}
Since $A$ is an elliptic singularity, we have $\br(I)\le 2$ by \cite{Ok}.
Hence (1) follows from the definitions of  $p_g$-ideals and elliptic ideals.
By \cite[3.8 and 3.10]{OWY1}, $p_g$-ideals of $A$ cannot be represented on the minimal resolution. Hence (2) follows.
\end{proof}

The elliptic sequence for elliptic singularities introduced by S.S.T.~Yau \cite{yau.max} is a very useful tool.

\begin{defn}
We define the {\em elliptic sequence} $\{Z_0, \dots, Z_m\}$ 
on $X$ as follows.  
Let $B_0=E$ and $Z_{0}=Z_{B_0}$. If $Z_0 E_{min}<0$, then the elliptic
 sequence is $\{Z_{0}\}$.
If $Z_0, \dots, Z_i$ are determined and $Z_{i}   E_{min}=0$, then define $B_{i+1}$ to be the
 maximal reduced connected cycle containing $\supp(E_{min})$ such that
 $Z_{i}  B_{i+1}=0$, and let $Z_{i+1}=Z_{B_{i+1}}$.
If we have $Z_{m}   E_{min}<0$ for some $m\ge 0$, then the elliptic
 sequence is $\{Z_{0}, \dots ,Z_{m}\}$.
\par 
For $0 \le t \le m$, we define cycles $C_t$ and $C'_t$ by
$$
C_t=\sum _{i=0}^tZ_{i} \quad \text{and} \quad
C_t'=\sum _{i=t}^mZ_{i}.
$$
We put $C_{-1}=C'_{m+1}=0$.
\end{defn}

\begin{ex}\label{e:244}
Assume that $\chara(k)\ne 2$.
Let $R=k[x,y,z]_{(x,y,z)}$
and put $A=R/(x^2+y^4+z^{4m+4})$
(cf. \exref{ex:Brieskorn} (2)(b)). 
Then $A$ is elliptic (see \cite[4.4]{OWY4}) and its resolution graph is as in \figref{fig:244}, where $E_{m}$ is an elliptic curve and other components are nonsingular rational curves.
Then $Z_i=E_m+\sum_{j=i}^{m-1}(E_{i}^{(1)}+E_{i}^{(2)})$.
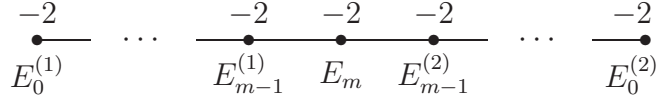
\begin{figure}[htb]
\begin{picture}(240,20)(30,15)
\put(35,25){\circle*{4}}
\put(35,25){\line(1,0){20}}
\put(75,25){\makebox(0,0){$\cdots$}}
\put(95,25){\line(1,0){20}}
\put(115,25){\circle*{4}}
\put(115,25){\line(1,0){90}}
\put(150,25){\circle*{4}}
\put(185,25){\circle*{4}}
\put(225,25){\makebox(0,0){$\cdots$}}
\put(245,25){\line(1,0){20}}
\put(265,25){\circle*{4}}
\put(35,35){\makebox(0,0){$-2$}}
\put(115,35){\makebox(0,0){$-2$}}
\put(150,35){\makebox(0,0){$-2$}}
\put(185,35){\makebox(0,0){$-2$}}
\put(260,35){\makebox(0,0){$-2$}}
\put(150,12){\makebox(0,0){$E_{m}$}}
\put(115,12){\makebox(0,0){$E_{m-1}^{(1)}$}}
\put(185,12){\makebox(0,0){$E_{m-1}^{(2)}$}}
\put(260,12){\makebox(0,0){$E_{0}^{(2)}$}}
\put(35,12){\makebox(0,0){$E_{0}^{(1)}$}}
\end{picture}
\caption{The resolution graph of $A=R/(x^2+y^4+z^{4m+4})$ }\label{fig:244}
\end{figure}
\end{ex}

\begin{prop}[cf. {\cite[\S 2]{nem.ellip}}]\label{p:seq}
We have the following properties$:$
\begin{enumerate}
 \item $B_0>B_1>\dots >B_m$, $Z_{m}=E_{min}$.
 \item $Z_{i} Z_{j}=0$ for $i\not=j$ and $-Z_0^2\ge \cdots \ge -Z_m^2$. 
 \item $C_t$ is anti-nef for 
 $0$  $\le t \le m$.
 \item For $0 \le t \le m$, $C'_t$ is the canonical cycle on $B_t$, namely, $(K_X+C'_t)E_i=0$ for every $E_i\le B_t$. 
 \item $\chi(C_t)=\chi(C'_t)=\chi(Z_{t})=0$ for $0\le t \le m$.
\end{enumerate}
\end{prop}
\par 
Since $A$ is Gorenstein, the canonical divisor $K_X$ is linearly equivalent to a cycle.
By \proref{p:seq} (4), we may assume that 
\begin{equation}\label{eq:Cm}
-K_X =C'_0=C_m.
\end{equation}
\par
The cycles  $C_t$ are characterized as follows.

\begin{prop} [{Tomari \cite[Theorem (6.4)]{T},
 N{\'e}methi \cite[2.13]{nem.ellip}}]
\label{p:char}
Let $D \ge 0$ be a cycle.
If $D \le C_m$ and $D$ is anti-nef, then $D=C_t$ for some $-1 \le t \le m$. 
\end{prop}

\begin{lem}[{\cite[2.13]{o.numGell}, cf. \cite[2.20]{nem.ellip}}]\label{l:fexists}
 For an integer  $t$
 such that $0 \le t \le m-1$, the  following conditions are equivalent$:$ 
\begin{enumerate}
 \item there exists a function $f\in H^0(\cO_X(-C_t))\setminus H^0(\cO_X(-C_{t+1}))$.
 \item $\cO_{C'_{t+1}}(-C_t) \cong \cO_{C'_{t+1}}$.
 \item $h^1(\cO_X(-C_t))=h^1(\cO_X(-C_{t+1}))+1$.
 \item $\dim _{k}H^0(\cO_X(-C_t))/H^0(\cO_X(-C_{t+1}))=1$.  
\end{enumerate}
If the conditions above are satisfied, then
 $h^1(\cO_X(-C_t))=h^1(\cO_{C'_{t+1}})$.
\end{lem}

The condition (2) of \lemref{l:fexists} says that  $\cO_{C'_{t+1}}(-(Z_0+ \cdots + Z_t)) \cong \cO_{C'_{t+1}}$. 
In fact, our main result in this section will need the fact that $\cO_{C'_{t+1}}(-Z_i) \cong \cO_{C'_{t+1}}(-Z_j)$ for $0 \le i \le j \le t$ and these invertible sheaves have a finite order in the Picard group of $C'_{t+1}$ (see \proref{p:ordp_g}, \ref{p:maxell} and \ref{p:bsz^2-1}).

\begin{defn}\label{d:Af}
  We define a set  $\cA_f$ by
\[
\cA_f=\defset{i}{0 \le i \le m-1, \; \text{$i$ satisfies the conditions in \lemref{l:fexists}}} \cup \{m\}.
\]
We put  $\beta=\min \cA_f$.
\end{defn}

\begin{rem}\label{r:Af}
 We have $\# \cA_f = p_g(A)$.
In fact, if $\cA_f=\{\beta_1, \dots, \beta_{\ell}\}$, $\beta_i < \beta _{i+1}$, then 
\begin{gather*}
\ell_A\left(A/ H^0( \cO_X( -C_{\beta_i} )) \right) = i, \\ 
\ell_A\left(A/ H^0( \cO_X( -C_{\beta_{\ell}} )) \right) =
\ell_A\left(A/ H^0( \cO_X( -C_{m} )) \right) = p_g(A),
\end{gather*}
where the last equality follows from \eqref{eq:Cm}.  
\end{rem}

\begin{defn}\label{d:fix}
For an element $f\in H^0(\cO_X(-E))=\m$, let $(f)_E$ denote the exceptional part of the divisor $\di_X(f)$; hence $\di_X(f)-(f)_E$ is the proper transform of the divisor $\di_{\spec(A)}(f)$ on $X$.
\par
Let $W>0$ be an anti-nef cycle.
A cycle $Y > 0$ is called the {\em fixed part} of $\cO_X(-W)$ if $Y$ is the maximal cycle such that $H^0(\cO_X(-W-Y))=H^0(\cO_X(-W))$; in this case, $(f)_E=W+Y$ for a general element $f\in H^0(\cO_X(-W))$.  
An irreducible component of the fixed part $Y$ is 
called a {\em fixed component}.
\par 
Note that if $\cO_X(-W)$ has no fixed component in a cycle $D>0$ and $WD=0$, then $\cO_D(-W)\cong \cO_D$.
\end{defn}

\begin{rem}
From the conditions (1) and (2) of \lemref{l:fexists}, for $i\in \cA_f$, $\cO_X(-C_i)$ has no fixed component on $B_{i+1}$.
\end{rem}

\begin{lem} \label{l:bs}
Let $D$ be a nef divisor on $X$. Then we have the following.
\begin{enumerate}
\item $\cO_X(-C_m)$ has no fixed component in $B_m$.
If  $Z_m^2= -1$, $\cO_{Z_m}(-C_m)$ is not generated.
\item If $DZ_m=1$, then $\cO_{Z_m}(D)$ is not generated.
\item  If $\cO_X(D)$ has no fixed component in $B_m$ and is not generated, then there exists $0 \le t \le m$ and 
a nonsingular point $Q$ of the scheme $C'_t$ 
such that $DC'_t=1$ and $Q$ is a base point of $\cO_X(D)$.
\par 
In particular, if $\cO_X(-C_t)$ with $t\in \cA_f$ is not generated, then $Z_t^2=-1$.
\end{enumerate}
\end{lem}
\begin{proof}
The claim (1) follows from (1) and (2) of \cite[3.3]{o.numGell}.
The claim (2) follows from \cite[4.23 Lemma]{Re}.
The assumption of (3) implies that $\cO_X(D)$ has only finitely many base points in $B_m$. Hence (3) follows from \cite[3.1]{o.numGell}.
\end{proof}

\begin{thm}\label{t:ZcA}
Let $Z>0$ be an anti-nef cycle on $X$.
Assume that $\cO_X(-Z)$ is generated 
 and  $\chi(Z)=0$.
Then $Z\in \defset{C_t}{ t\in \cA_f}$. \par
\end{thm}
\begin{proof}
 Note that $\ell_A(A/H^0(\cO_X( -C_m ) ) )=p_g(A)$ and that  
$\ell_A(A/I_Z)=p_g(A)-q(I_Z) \le p_g(A)$ (cf. the proof of \thmref{Main}).
\par
First we consider the case that $Z \ge C_m$. If $Z = C_m$, we have nothing to say. 
Hence assume that $Z = C_m +Y$ with $Y > 0$. Then since $\ell_A(A/I_Z)  
\ge \ell_A(A/H^0(\cO_X( -C_m ) ) )=p_g(A)$, we have $\ell_A(A/I_Z) = p_g(A)$ 
and $q(I_Z) =0$.
Moreover,
since $\HH^0(X, \cO_X(-C_m)) 
= \HH^0(X, \cO_X( -Z))$, 
$Y$ consists of some fixed components of $\cO_X(-C_m)$.  
Also, we have 
\[0 = \chi(Z) = \chi(C_m) + \chi(Y) - Y C_m. \]
Since $\chi(C_m)=0$ and $C_m=-K_X$ is anti-nef (cf. \proref{p:seq} (3)), we have $\chi(Y) = Y C_m \le 0$.  
But from Lemma \ref{l:bs} (1), 
we must have $\chi(Y) > 0$. 
 This contradiction shows that $Z = C_m$.
\par 
Next, we will show that $Z < C_m$ if $Z\ne C_m$.  
\par
Assume that $Z \not\le C_m$ and let $G=\gcd(Z, C_m)$, namely, $G$ is the maximal 
cycle with the property that $G\le Z$ and 
$G\le C_m$.
 Since both $Z, C_m$ are anti-nef, so is $G$.
 Put $Z = G + F$. 
Then $G=C_t$ for some $0\le t < m$ by \proref{p:char}, and $\chi(F) > 0$ by \proref{p:ellchi} (1) since $F\not \ge Z_m$.
Hence
\[\chi(Z) = \chi(F+G) = \chi(F) + \chi(G) - FG  = \chi(F) - FG > 0.
 \]  
However, this contradicts our assumption
 that $\chi(Z) = 0$. 
 Hence we have $Z \le C_m$ and our conclusion follows from 
 Lemma \ref{p:char}.
\end{proof}

\begin{cor}\label{c:ellGG2}
Assume that $A$ is an elliptic Gorenstein singularity.
\begin{enumerate}
\item If $t\in \cA_f$ and $-Z_t^2\ge 2$, then $\cO_X(-C_t)$ is generated and $\overline{G}(I_{C_t})$ is Gorenstein.  
\item Let  $\zeta$ be the number of elliptic ideals $I \subset A$ such that $\overline{G}(I)$ is Gorenstein.
Then $\zeta \le p_g(A)$, and the equality holds if and only if $-Z_m^2\ge 2$.
\end{enumerate}
\end{cor}

\begin{proof}
By \thmref{Main}, \lemref{l:repMin}, and \thmref{t:ZcA}, any elliptic ideal $I$ such that $\overline{G}(I)$ is Gorenstein is represented by some $C_i$ on the minimal resolution.
Conversely, if an ideal $I$ is represented by some $C_i$, then $I$ is elliptic by \lemref{l:br=2} and $\overline{G}(I)$ is Gorenstein by \thmref{Main}.
 If $t\in \cA_f$ and $-Z_t^2\ge 2$, then $\cO_X(-C_t)$ is generated by \lemref{l:bs} (3). Hence we obtain (1).
For (2), note also that  $\# \cA_f = p_g(A)$ by \remref{r:Af} and that $-Z_t^2 \ge 2$ for all $t\in \cA_f$ if and only if $-Z_m^2\ge 2$ by \proref{p:seq} (2).
\end{proof}

\begin{defn}
A numerically Gorenstein normal two-dimensional elliptic singularity is called a {\em maximally elliptic} singularity if its geometric genus equals to the length of the elliptic sequence ($m+1$, in our case).
It is known that every maximally elliptic singularity 
is Gorenstein (\cite{yau.max}).
\end{defn}

\begin{ex}
Let $A$ be as in \exref{e:244}. Then $A$ is maximally elliptic and $\zeta = p_g(A) = m+1$.
\end{ex}

\par 
We have to classify the cycles $C_t$ ($t\in \cA_f$) such that $\cO_X(-C_t)$ is generated. \lemref{l:bs} tells us that if $\cO_X(-C_t)$ is not generated, then $Z_t^2=-1$; however the converse does not hold in general (cf. \exref{e:NoMaxEll}).
To analyze the generation of the invertible sheaves $\cO_X(-C_t)$, we have to know at first the combinatorial structure of the cycles $Z_t$ with $Z_t^2=-1$,
 and then some results in \cite{o.numGell} which depends on the structure of the singularity.

\begin{lem}[{\cite[3.5]{o.numGell}}]\label{l:z2=-1}
Assume that $0\le j \le m$ and $Z_{j}^2=-1$.
For $i\ge j$, let $F_i$ denote the unique irreducible component of $Z_i$ such that $Z_i^2=Z_iF_i=-1$ (see \proref{p:seq} (2)). 
\begin{enumerate}
\item Assume that $j\le i<m$.
Then $Z_{i}-Z_m=F_{m-1}+\cdots + F_i$ and this is a chain of $(-2)$-curves (see \figref{fig:Z2-1}).
Therefore, $C_mF_i=-K_XF_i=0$.

\item If $j>0$, there exists an unique irreducible component $\t F\le E-B_j$ such that $\t FB_j>0$.
Then $\t F^2=-2$ if and only if $Z_{j-1}^2=-1$.

\item Assume that $Z_{j-1}^2=-1$.
Let $D>0$ be a cycle such that $DC=0$ for every irreducible component $C\subset B_j$ and that the coefficient of $F_{j-1}$ in $D$ is $n$.
Then $\supp(D - nZ_{j-1})$ has no component of $B_{j-1}$.
\end{enumerate}
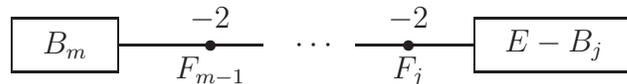
\begin{figure}[htb]
\begin{picture}(250,25)(70,10)
\put(75,15){{\framebox(40,20){$B_m$}}}
\put(150,25){\circle*{4}}
\put(225,25){\circle*{4}}
\put(115,25){\line(1,0){55}}
\put(205,25){\line(1,0){20}}
\put(150,35){\makebox(0,0){$-2$}}
\put(190,25){\makebox(0,0){$\cdots$}}
\put(225,35){\makebox(0,0){$-2$}}
\put(150,15){\makebox(0,0){$F_{m-1}$}}
\put(225,15){\makebox(0,0){$F_j$}}
\put(225,25){\line(1,0){25}}
\put(250,15){{\framebox(60,20){$E-B_j$}}}
\end{picture}
\caption{The resolution graph with $Z_j^2=-1$}\label{fig:Z2-1}
\end{figure}
\end{lem}

\begin{lem}[{\cite[3.6]{o.numGell}}]
\label{l:3free}
 For $t\in \cA_f \setminus \{m\}$ with $Z_t^2=-1$,  the following conditions are equivalent$:$
\begin{enumerate}
 \item $\cO_X(-C_t)$ is generated$;$
 \item $H^0(\cO_X(-C_t-Z_{t}))\subsetneqq H^0(\cO_X(-C_{t+1}))$$;$
 \item there exists a function $g \in H^0(\cO_X(-C_{t+1}))$ which satisfies
       $(g)_E=C_{t+1}$ outside $B_{t+2}$, where  we put
       $B_{m+1}=\emptyset$.
\end{enumerate}
\end{lem}

Actually, the proofs of some results in \cite{o.numGell} rely on the complex analytic setting; however, 
most of them hold in our algebraic situation. 
For example, \lemref{l:torsion} below enables us to prove Proposition 5.12 and Theorem 5.13 of \cite{o.numGell} in the characteristic zero case, and therefore, we obtain the following.

\begin{prop}[cf. {\cite[5.12, 5.13]{o.numGell}}\footnote{The invariant $\alpha$ in \cite{o.numGell} is zero since $A$ is Gorenstein.}]\label{p:ordp_g}
Assume that $\chara(k)=0$. 
Let $\gamma=\beta+1$ (see \defref{d:Af}).
Then we have the following.
\begin{enumerate}
\item  For $0 \le t <u\le m$,  the order of $\cO_{C'_u}(Z_0)$ in the Picard group $\pic C'_u$ is $\gamma$  
and $\cO_{C'_u}(Z_0)\cong \cO_{C'_u}(Z_t)$. 
\item  $\gamma \mid m$ and $p_g\X=m/\gamma +1$.
\item $\cA_f=\defset{\beta+i\gamma}{0 \le i < m/\gamma} \cup \{m\}$.
\end{enumerate}
\end{prop}

\begin{lem}[{cf. \cite[3.7]{o.numGell}}]
\label{l:torsion}
Assume that $\chara(k)=0$. 
Let $D$ be a cycle such that $D\ge Z_m$ and $\supp(D)$ is connected.  
We denote by $\pic^0 D$ the subgroup of $\pic D$ which consists of numerically trivial elements.
Let $\phi: \pic^0 D \to \pic ^0 Z_m$ be the homomorphism obtained by the restriction.
Then we have the following.
\begin{enumerate}
\item $\Ker \phi$ is torsion free.
\item If $\sigma \in \pic^0 D$ is a torsion element, then $\ord(\sigma)=\ord(\phi(\sigma))$.
\end{enumerate}
\end{lem}

\begin{proof}
(1) By \cite[Lemma 1.4]{Ar-num}, the kernel of the restriction map $\phi':\pic D \to \pic D_{red}$ is torsion free.
If $H^1(\cO_{D_{red}})=0$, then $\pic^0 D_{red}$ is trivial (cf. \cite[11.1]{Li}) and $\pic^0 D$ is torsion free.
Assume $H^1(\cO_{D_{red}})\ne 0$. Since $\chi(D_{red})=0$ and $Z_m$ is the minimally elliptic cycle, we have $Z_m\le D_{red}$ by \proref{p:ellchi}. 
Let $L$ be an invertible sheaf on $D_{red}$ corresponding to an element of $\pic^0 D_{red}$. 
We show that $L$ is trivial if and only if so is $L|_{Z_m}$.
Let $C=D_{red}-Z_m$.
We may assume $C>0$.
Then we have the exact sequence
\[
0 \to L\otimes \cO_C(-Z_m) \to L \to L|_{Z_m} \to 0.
\]
Since $C'Z_m = 1$ and $\chi(C')=1$ for every connected component $C'\le C$ by \proref{p:ellchi}, we have $H^i(L\otimes \cO_C(-Z_m))=0$ for $i=0,1$ (cf. \cite[1.4]{Rh}), and hence $H^0(L)=H^0(L|_{Z_m})$. 
Since an invertible sheaf on a connected reduced cycle is trivial if it has non-zero global sections (cf. \cite[3.11]{Re}), $L$ is trivial if and only if so is $L|_{Z_m}$.
Hence  $\Ker \phi = \Ker ( \pic^0 D \to \pic ^0 D_{red} ) \subset \Ker \phi'$. 
\par 
(2) If $n=\ord(\phi(\sigma))$, then $n\sigma$ is a torsion element in $\Ker \phi$. Hence $n\sigma$ is trivial, and $\ord(\sigma)=n$.
\end{proof}
\par 
From \remref{r:Af} and \proref{p:ordp_g}, we have the following.

\begin{prop}\label{p:maxell}
For a Gorenstein elliptic singularity $A$ with $\chara(k)=0$, 
the following are equivalent.
\begin{enumerate}
\item $A$ is maximally elliptic.
\item $\beta=0$.
\item $\gamma=1$.
\item $\cA_f=\{0, 1, \dots, m\}$.
\end{enumerate}
\end{prop}

\begin{prop}\label{p:bsz^2-1}
 Assume that $A$ is a Gorenstein elliptic singularity.
\begin{enumerate}
\item Assume that $\chara(k)=0$. 
 If there exists $t\in \cA_f\setminus \{m\}$ such that $\cO_X(-C_t)$ is not generated, then $A$ is maximally elliptic and $Z_t^2=-1$.
\item If $A$ is maximally elliptic, then $\cO_X(-C_t)$ is not generated for every $t\in \cA_f$ with $Z_t^2=-1$.
\end{enumerate} 
\end{prop}
\begin{proof}
(1) Assume that $t\in \cA_f\setminus \{m\}$ and $\cO_X(-C_t)$ is not generated.
We have $Z_i^2=-1$ for $i\ge t$ by \proref{p:seq} (2) and  \lemref{l:bs} (3), and $C_t+Z_{t}=C_{t+1}+F_t$, where $F_t$ is as in \lemref{l:z2=-1}. 
If $\cO_X(-C_m)$ has no fixed component, then there exists $g\in H^0(\cO_X(-C_m))$ such that $(g)_E=C_m$, and thus the condition (3) of \lemref{l:3free} is satisfied; 
however, it contradicts the assumption that $\cO_X(-C_t)$ is not generated.
Therefore, $\cO_X(-C_m)$ has a fixed part $F$ which contains $F_t$
because $H^0(\cO_X(-C_m)) \subset H^0(\cO_X(-C_{t+1})) =H^0(\cO_X(-C_{t+1}-F_t))$ by \lemref{l:3free}, 
 and $F$ has no components of $B_m$ by \lemref{l:bs} (1).
Since $\cO_X(-C_m-F)$ has no fixed component, $C_m+F$ is anti-nef.
If $F$ does not intersect $Z_m$, then we can take $F_i$ of \lemref{l:z2=-1} with $i< m$ so that  $F_i\not\le F$ and $F_i$ intersects $F$; however, it induces a contradiction that $(C_m+F)F_i = FF_i>0$.
Hence $F$ intersects $Z_m$ and $(C_m+F)Z_m = Z_m^2+FZ_m = -1+FZ_m=0$, because $FZ_m>0$.
Therefore, taking $g\in H^0(\cO_X(-C_m-F))$ such that $(g)_E=C_m+F$, we have $\cO_{Z_m}(C_m+F)\cong \cO_{Z_m}(\di (g)) \cong \cO_{Z_m}$.
Since $\cO_{Z_m}(-C_m+Z_m)\cong \omega_{Z_m}\cong \cO_{Z_m}$ and 
$\cO_{Z_m}(-Z_{m-1})\cong \cO_{Z_m}(-Z_m-F)$ (cf. \lemref{l:z2=-1}), we have 
\[
\cO_{Z_m}(-Z_{m-1})\cong \cO_{Z_m}(-Z_m-F-C_m+Z_m)\cong \cO_{Z_m}(-C_m-F)\cong \cO_{Z_m}.
\]
It follows from \proref{p:ordp_g} (1) that $\gamma=1$.
Therefore, $A$ is maximally elliptic by \proref{p:maxell}.
\par
(2) 
Assume that $A$ is maximally elliptic.
Since $\cA_f=\{0, 1, \dots, m\}$ by \proref{p:maxell}, $\cO_{C'_j}(-Z_i)\cong \cO_{C'_j}$ for $0 \le i< j \le m$ by \lemref{l:fexists}.
It follows from \lemref{l:bs} (1) that $\cO_X(-C_m)$ is not generated if $Z_m^2= -1$. 
Assume that $Z_{m-1}^2=-1$.
Then $Z_m^2=-1$ by \proref{p:seq} (2).
Since  $Z_m+F_{m-1}=Z_{m-1}$, we have 
$\cO_{Z_m}(-Z_m-F_{m-1})\cong
\cO_{Z_m}(-Z_{m-1})\cong  \cO_{Z_m}$;
this implies that the point $B_m\cap F_{m-1}$ is the base point of $\cO_{Z_m}(-C_m)$.
If there exists $g\in H^0(\cO_X(-C_m))$ such that $(g)_E=C_m+Y$, where $Y\ge 0$ and $Y$ has no component of $Z_{m-1}$, 
then the divisor $H:=\di_X(g)-C_m-Y$ satisfies $HF_{m-1}= -Y F_{m-1} \le0$; it contradicts that the base point of $\cO_{Z_m}(-C_m)$ is $B_m\cap F_{m-1}$.
Thus we see that $F_{m-1}$ is a fixed component of $\cO_X(-C_m)$.
By \lemref{l:3free}, $\cO_X(-C_{m-1})$ is not generated, because (3) of the lemma is not satisfied for $t=m-1$.
Since $\cO_{C'_{m-1}}(-C_{m-1}-F_{m-2}) = \cO_{C'_{m-1}}(-C_{m-2}-Z_{m-2}) \cong \cO_{C'_{m-1}}$, 
the similar argument as above implies that $F_{m-2}$ is a fixed component of $\cO_X(-C_{m-1})$.
Inductively, we obtain that $\cO_X(-C_t)$ is not generated for $0\le t < m$ with $Z_t^2=-1$.
\end{proof}

\begin{thm}\label{t:ellGG}
Assume that $A$ is an elliptic Gorenstein singularity.
Let $I$ be an elliptic ideal, namely, $\br(I)=2$.
\begin{enumerate}
\item Assume that $A$ is maximally elliptic.
Then $\overline{G}(I)$ is Gorenstein if and only if 
$I$ is represented by $C_t$  with $t\in \cA_f$ such that $-Z_t^2\ge 2$.
\item Assume that  $\chara(k)=0$ and $A$ is not maximally elliptic.
Then $\overline{G}(I)$ is Gorenstein if and only if 
$I$ is represented by $C_t$ with $t\in \cA_f \setminus \{m\}$ or $C_m$ with $-Z_m^2\ge 2$.  
\end{enumerate}
\end{thm}
\begin{proof}
As mentioned in the proof of \corref{c:ellGG2}, any elliptic ideal $I$ such that $\overline{G}(I)$ is Gorenstein is represented by some $C_i$ on the minimal resolution.
Therefore, the classification of such ideals correspond to that of cycles $C_i$ such that $\cO_X(-C_i)$ is generated; it is done by a combination of \lemref{l:bs} and \proref{p:bsz^2-1}.  
\end{proof}

\begin{cor}\label{c:MaxEllG}
Assume that $A$ is a maximally elliptic singularity with  $Z_0^2=-1$.
If $I \subset A$ is an $\m$-primary integrally closed ideal and $\overline{G}(I)$ is Gorenstein, then $\br(I)=1$, namely, $I$ is a $p_g$-ideal.
\end{cor}
\begin{proof}
We have $\br(I) \le 2$ by \cite[3.3]{Ok}. 
By \proref{p:seq}, $Z_t^2=-1$ for $0\le t \le m$. Therefore, the assertion immediately follows from \thmref{t:ellGG}.
\end{proof}

\begin{rem}[{Cf. \cite[4.10]{Ok}}]
In general, for a two-dimensional normal Gorenstein singularity $(A, \m)$ with $p_g(A)>0$, 
$\m$ is a $p_g$-ideal if and only if $A$ is a maximally elliptic Gorenstein singularity with $Z_0^2=-1$.
\end{rem}

%

The following example shows that the elliptic ideals $I \subset A$ such that $\overline{G}(I)$ is Gorenstein cannot be characterized by their resolution graph.

\begin{ex}
[{cf. \cite[4.7]{K}}]
\label{e:NoMaxEll}
Let $n$ be a positive integer and consider the polynomials
\[
f=x^2+z(z^{4n+2}+y^4) \quad \text{and} \quad g=x^2+y^3+z^{6(2n+1)},
\]
which are weighted homogeneous polynomials with respect to the weights $(4n+3,2n+1,2)$ and $(3(2n+1),2(2n+1),1)$, respectively.
Let $R=k[x,y,z]_{(x,y,z)}$, $A=R/(f)$ and $B=R/(g)$.
These singularities have the same resolution graph as in \figref{fig:23(12)}, where $g(E_{2n})=1$, $g(E_i)=0$ for $0\le i \le 2n-1$.

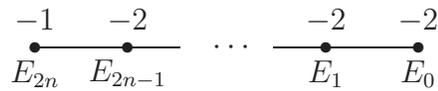
\begin{figure}[htb]
\begin{picture}(160,20)(110,15)
\put(115,25){\circle*{4}}
\put(150,25){\circle*{4}}
\put(225,25){\circle*{4}}
\put(260,25){\circle*{4}}
\put(115,25){\line(1,0){55}}
\put(190,25){\makebox(0,0){$\cdots$}}
\put(205,25){\line(1,0){20}}
\put(115,35){\makebox(0,0){$-1$}}
\put(150,35){\makebox(0,0){$-2$}}
\put(225,35){\makebox(0,0){$-2$}}
\put(260,35){\makebox(0,0){$-2$}}
\put(115,15){\makebox(0,0){$E_{2n}$}}
\put(150,15){\makebox(0,0){$E_{2n-1}$}}
\put(225,15){\makebox(0,0){$E_1$}}
\put(260,15){\makebox(0,0){$E_0$}}
\put(225,25){\line(1,0){35}}
\end{picture}
\caption{The resolution graph of $A=R/(f)$ and $B=R/(g)$}\label{fig:23(12)}
\end{figure}
\par 
It is easy to check that these singularities are Gorenstein elliptic singularities such that $Z_i=\sum_{j=i}^{2n} E_j$, $Z_i^2=-1$, and $C_iE_j=-\delta_{ij}$, where $\delta_{ij}$ denotes the Kronecker delta, for $0 \le i \le 2n$.
Using \thmref{t:pg-f}, we easily see that $p_g(A)=n+1$ and $p_g(B)=2n+1$. 
\par 
Then $B$ is a maximally elliptic singularity, and hence there is no elliptic ideal $I$ such that $\overline{G}(I)$ is Gorenstein by \corref{c:MaxEllG}.
\par
Assume that $\chr(k)\ne 2$.
Let  $\cal G$ be the set of elliptic ideals $I \subset A$ such that $\overline{G}(I)$ is Gorenstein.
We will show that $\cal G = \defset{(x,y,z^{j})}{1 \le j \le n}$.
First, let us identify the divisors of $x, y, z$ on $X$.
For a divisor $D$ on the resolution space $X$, we denote by $c(D)$ the coefficient of $E_{2n}$ in $D$.
For example, $c(C_{i})=i+1$. 
Let $H\subset X$ be the proper transform of the curve $\spec A/(x,z)$.
Since $A/(z)=k[x,y]_{(x,y)}/(x^2)$ and $c(\di_X(z))=\deg(z)=2$, we have $\di_X(z)=2Z_0+2H$;  we also have  $HE = H E_0 = 1$.
Recall that $E_{2n}$ is isomorphic to the curve $V(f)\subset \PP(4n+3,2n+1,2)$. The cone over $V(f)$ has an isolated singularity at the vertex and the localization of its local ring coincides with $A$.
We take points $p_0, p_1, p_2, p_3 \in V(f)$ such that 
\[
p_0=(0,1,0), \ \ p_1=(0,(\sqrt{-1})^{n},\sqrt{-1}), \ \ p_2=(0,(\sqrt{-1})^{n+1},\sqrt{-1}), \ \ p_3=(1,0,-1).
\] 
Then these points are distinct and 
\[
\{p_0, p_1, p_2\}=V(f,x) , \quad \{p_3\}=V(f,y), \quad 
\{p_0\}=V(f,z) \subset \PP(4n+3,2n+1,2).
\]
Let $H_i\subset X$ be the proper transform of the curve in $\spec A$ corresponding to the cone over the point $p_i$. 
Clearly, $H=H_0$, $H_i\cap E_{2n}\ne \emptyset$ and $H_i\cap E_{2n-1}=\emptyset$ for $i=1,2,3$.
Since $\deg(y)=2n+1$ and $H_3\equiv -(H_3E_{2n})C_{2n}$,  we have  $H_3E_{2n}=1$ and $\di_X(y)=C_{2n}+H_3$.
Similarly, we have $\di_X(x)=2C_{2n}+Z_0+H_0+H_1+H_2$.
Thus, for $1\le j \le n$, $\gcd((y)_E, (z^{j})_E)=\gcd(C_{2n},2j Z_0)=C_{2j-1}$.
Moreover, $y$ (resp. $z^j$) generates $\cO_X(-C_{2j-1})$ at every point of $\sum_{i=0}^{2j-1} E_i$ (resp. $\supp(Z_{2j})\setminus E_{2j-1}$).
In particular, $2j-1\in \cA_f$.
Let $I_j=I_{C_{2j-1}}$.
Since $p_g(A)=n+1$, it follows from \remref{r:Af} that $\cA_f=\defset{2j-1}{1 \le j \le n} \cup \{2n\}$ and $\ell_A\left(A/ I_j \right) = j$.
As noted in the proof of \thmref{t:ellGG}, the ideals in $\cal G$ correspond to the cycles $C_i$ such that $\cO_X(-C_i)$ is generated. Hence $\cal G=\defset{I_j}{1 \le j \le n}$.
Since  $(x, y, z^{j}) \subset I_j$ and $\dim _k k[x,y,z]/(f,x, y, z^{j}) =j=\ell_A\left(A/ I_j \right)$, we have $I_j=(x,y,z^{j})$.
\end{ex}

\par
By \corref{c:MaxEllG}, \corref{c:ellGG2} and the argument in \exref{e:NoMaxEll}, we have the following.

\begin{ex}\label{e:BrEll}
Let $R=k[x,y,z]_{(x,y,z)}$ and let $m \ge 0$ be an integer. 
In each of the following cases, $A$ is a maximally elliptic singularity with $p_g(A)=m+1$, $\cA_f=\{0,1, \dots, m\}$, the minimally elliptic cycle $Z_m$ is an elliptic curve, and $Z_0 = E$. 
Let  $\cal G$ be the set of elliptic ideals $I \subset A$ such that $\overline{G}(I)$ is Gorenstein.
\begin{enumerate}
\item If we put $A=R/(x^2+y^3+z^{6(m+1)})$, then $-Z_m^2=-1$ and $\cal G=\emptyset$. 
The exceptional set $E$ consists of $Z_m$ and a chain of $(-2)$-curves $\sum_{j=0}^{m-1} E_j$ such that $Z_m E_{m-1} = E_{j}E_{j-1}= 1$ 
($1\le j \le m-1$) (cf. \figref{fig:23(12)}).

\item If we put $A=R/(x^2+y^4+z^{4(m+1)})$, then $-Z_m^2=-2$ and 
\[
\cal G=\defset{(x,y,z^i)}{1 \le i \le m+1}.
\]
The exceptional set $E$ is described in \exref{e:244} and  $(x,y,z^i) = I_{C_{i-1}}$.

\item If we put $A=R/(x^3+y^3+z^{3(m+1)})$, then $-Z_m^2=-3$ and 
\[
\cal G=\defset{(x,y,z^i)}{1 \le i \le m+1}.
\]
The exceptional set $E$ consists of an elliptic curve $Z_m$ and three chains of $(-2)$-curves $\sum_{j=0}^{m-1} E_j^{(s)}$ ($s=1,2,3$) and $(x,y,z^i) = I_{C_{i-1}}$ as above.
\end{enumerate}
\end{ex}

\begin{quest} \label{openQ}
Let $A$ be an excellent normal Gorenstein local domain of $\dim A=2$. 
Then is the set of elliptic ideals $I=I_Z$ for which $\overline{G}(I)$ is Gorenstein  a finite set?
\end{quest}

\begin{acknowledgement} 
 The authors thank the referee for careful reading of the paper and 
helpful comments.
\end{acknowledgement}




\begin{thebibliography}{1}

\bibitem[Ar]{Ar-num}
Michael Artin, \emph{Some numerical criteria for contractability of curves on
  algebraic surfaces}, Amer. J. Math. \textbf{84} (1962), 485--496.

  


 

  

\bibitem[GI]{GI}
Shiro Goto and Shin-ichiro Iai, \emph{Embeddings of certain graded 
rings into their canonical modules}, J. Algebra \textbf{228} (2000),
 no.~1, 377--396.

\bibitem[GIW]{GIW}
Shiro Goto, Sin-ichiro Iai, and Kei-ichi Watanabe, \emph{Good ideals in
  {G}orenstein local rings}, Trans. Amer. Math. Soc. \textbf{353} (2001),
  no.~6, 2309--2346. 


\bibitem[GN]{GN}
Shiro Goto and Koji Nishida, \emph{The {Cohen}-{Macaulay} and {Gorenstein}
  {Rees} algebras associated to filtrations}, Mem. Am. Math. Soc., vol. 526,
  Providence, RI: American Mathematical Society (AMS), 1994.

\bibitem[GW]{GW}
Shiro Goto and Kei-ichi Watanabe, \emph{On graded rings. {I}}, J. Math. Soc.
  Japan \textbf{30} (1978), no.~2, 179--213.




\bibitem[Hun]{Hun} 
Craig Huneke, \emph{Hilbert functions and symbolic powers}, Michigan Math. J.
  \textbf{34} (1987), no.~2, 293--318.

\bibitem[It]{It1} 
Shiroh Itoh, \emph{Integral closures of ideals generated by regular sequences},
  J. Algebra \textbf{117} (1988), no.~2, 390--401.



\bibitem[Kr]{karras}
Ulrich Karras, \emph{Local cohomology along exceptional sets}, Math. Ann.
  \textbf{275} (1986), 673--682.

\bibitem[Kt]{kato}
Masahide Kato, \emph{Riemann-{R}och theorem for strongly pseudoconvex manifolds
  of dimension {$2$}}, Math. Ann. \textbf{222} (1976), no.~3, 243--250.

\bibitem[Ko]{K}
Kazuhiro Konno,   \emph{On the fixed loci of the canonical systems over normal surface
  singularities}, Asian J. Math. \textbf{12} (2008), no.~4, 449--464.

\bibitem[La]{La} 
Henry~B. Laufer,
 \emph{On minimally elliptic singularities}, Amer. J. Math.
  \textbf{99} (1977), no.~6, 1257--1295.

\bibitem[Li]{Li} 
Joseph Lipman, \emph{Rational singularities, with applications to algebraic
  surfaces and unique factorization}, Inst. Hautes \'Etudes Sci. Publ. Math.
  (1969), no.~36, 195--279.






\bibitem[Ne]{nem.ellip}
Andr\'{a}s N\'{e}methi, \emph{``{W}eakly'' elliptic {G}orenstein singularities
  of surfaces}, Invent. Math. \textbf{137} (1999), no.~1, 145--167.



\bibitem[Ok1]{o.numGell}
Tomohiro Okuma, \emph{Numerical {G}orenstein elliptic singularities}, Math. Z.
  \textbf{249} (2005), no.~1, 31--62.


\bibitem[Ok2]{Ok}
\bysame, 
 \emph{Cohomology of ideals in elliptic surface singularities}, 
 Illinois J. Math. \textbf{61} (2017), no.~3--4, 259--273.

\bibitem[ORWY]{ORWY}
Tomohiro Okuma, Maria~Evelina Rossi, Kei-ichi Watanabe, and Ken-ichi Yoshida,
  \emph{Normal {Hilbert} coefficients and elliptic ideals in normal
  two-dimensional singularities}, Nagoya Math. J. \textbf{248} (2022),
  779--800.

\bibitem[OWY1]{OWY1}
Tomohiro Okuma, Kei-ichi Watanabe, and Ken-ichi Yoshida, \emph{Good ideals and
  {$p_g$}-ideals in two-dimensional normal singularities}, Manuscripta Math.
  \textbf{150} (2016), no.~3-4, 499--520.

\bibitem[OWY2]{OWY2}
\bysame,
 \emph{Rees algebras and {$p_g$}-ideals in a two-dimensional normal
  local domain}, Proc. Amer. Math. Soc. \textbf{145} (2017), no.~1, 39--47.



\bibitem[OWY3]{OWY4}
\bysame,
\emph{Normal reduction numbers for normal surface singularities with
  application to elliptic singularities of {B}rieskorn type}, Acta Math.
  Vietnam. \textbf{44} (2019), no.~1, 87--100.


\bibitem[OWY4]{OWY5}
\bysame,
\emph{The normal reduction number of two-dimensional cone-like
  singularities}, Proc. Amer. Math. Soc. \textbf{149} (2021), no.~11,
  4569--4581.




\bibitem[Re]{Re}
Miles Reid,
\emph{Chapters on algebraic surfaces}, Complex algebraic geometry,
  IAS/Park City Math. Ser., vol.~3, Amer. Math. Soc., Providence, RI, 1997,
  pp.~3--159.

\bibitem[Rh]{Rh}
A.~R{\"o}hr, \emph{A vanishing theorem for line bundles on resolutions of  surface singularities}, Abh. Math. Sem. Univ. Hamburg \textbf{65} (1995),
  215--223.


 



\bibitem[To]{T}
Masataka Tomari, \emph{A {$p_g$}-formula and elliptic singularities}, Publ. Res. Inst.
  Math. Sci. \textbf{21} (1985), no.~2, 297--354.

\bibitem[TW]{TW}
Masataka Tomari and Kei-ichi Watanabe, \emph{Filtered rings, filtered
  blowing-ups and normal two-dimensional singularities with ``star-shaped''
  resolution}, Publ. Res. Inst. Math. Sci. \textbf{25} (1989), no.~5, 681--740.  

\bibitem[VV]{FormRing}
Paolo Valabrega and Giuseppe Valla, \emph{Form rings and regular sequences},
  Nagoya Math. J. \textbf{72} (1978), 93--101. 

\bibitem[Wa]{W} 
Philip Wagreich, \emph{Elliptic singularities of surfaces}, Amer. J. Math.
  \textbf{92} (1970), 419--454.

\bibitem[Wt1]{Wt}
Kei-ichi Watanabe, \emph{Some remarks concerning {D}emazure's construction of
  normal graded rings}, Nagoya Math. J. \textbf{83} (1981), 203--211.

\bibitem[Wt2]{Wt_rat}
\bysame, \emph{Rational singularities with k*-action}, Commutative algebra,
  {Proc}. {Conf}., {Trento}/{Italy} 1981, {Lect}. {Notes} {Pure} {Appl}.
  {Math}. 84, 339-351 (1983).

\bibitem[WY]{WY} 
Kei-ichi Watanabe and Ken-ichi Yoshida, \emph{Hilbert-{Kunz} multiplicity,
  {McKay} correspondence and good ideals in two-dimensional rational
  singularities}, Manuscr. Math. \textbf{104} (2001), no.~3, 275--294.

\bibitem[Yau]{yau.max}
Stephen Shing~Toung Yau,  \emph{On maximally elliptic singularities}, Trans. Amer. Math. Soc.
  \textbf{257} (1980), no.~2, 269--329.

\end{thebibliography}
\end{document}